\documentclass{amsart}
\usepackage{amsmath,amsthm,amssymb,graphicx}
\usepackage{color}
\newtheorem{thr}{Theorem}[section]
\newtheorem{lem}[thr]{Lemma}

\theoremstyle{definition}

\newtheorem{quest}[thr]{Problem}

\theoremstyle{remark}

\numberwithin{equation}{section}

\begin{document}

\title{Chess God's number grows exponentially}

\author{Yaroslav Shitov}
\address{National Research University Higher School of Economics, 20 Myasnitskaya Ulitsa, Moscow 101000, Russia}
\email{yaroslav-shitov@yandex.ru}


\begin{abstract}
We give an example of two $n\times n$ chess positions, $A$ and $B$, such that
(1) there is a sequence $\sigma$ of legal chess moves leading from $A$ to $B$;
(2) the length of $\sigma$ cannot be less than $\exp \Theta(n)$.
\end{abstract}

\maketitle

\section{Introduction}


This note presents a contribution in the complexity theory of \textit{puzzles}, or \textit{one-player games},
which can be considered directed graphs with vertices called \textit{positions} and arcs called \textit{moves}.
A player is given a pair of positions and needs to transform one to the other using a sequence of moves.
Well known examples of puzzles include Rubik's cube, Fifteen game, computer simulations like Atomix and Sokoban,
and other games. An important invariant of a puzzle is its \textit{diameter}, that is, the greatest possible distance between
a pair of positions, with distance being the length of a shortest sequence of moves transforming one
to the other. This invariant, hard to be calculated in general~\cite{15}, is sometimes referred to as \textit{God's number}.
For instance, the Rubik's cube God's number has recently been found by an extensive computer search~\cite{Rubik}.

If the diameter of a certain puzzle has an upper bound polynomial in its size,
the player can decide whether a solution exists in non-deterministic polynomial time
just by executing every possible sequence of moves. This is indeed the case for $n\times n$
generalizations of Fifteen game~\cite{15}, Rubik's cube~\cite{Rubik2}, and some other puzzles. On the other hand,
the Sokoban God's number has no polynomial upper bound~\cite{SB}, and this game turns out to be PSPACE-complete.

The order of growth is still unknown for diameters of several classical
puzzles, and this note aims to treat the problem for $n\times n$ chess.
A related problem has been mentioned in 1981 by Fraenkel and
Lichtenstein~\cite{Chess}, who noted that \textit{reachability}
of a given position from another one may be not quite infeasible.
Chow gives the following formulation of this problem.

\begin{quest}\cite{MO}\label{probexp}
Does there exist an infinite sequence $(A_n,B_n)$ of pairs of chess positions on an $n\times n$ board such that
the minimum number of legal moves required to get from $A_n$ to $B_n$ is exponential in $n$?
\end{quest}

Our note gives a positive solution of this problem, and we use the
following notation throughout. As said above, a \textit{puzzle} is a
directed graph with vertices called \textit{positions} and arcs called \textit{moves}.
We say that a sequence $\sigma=\sigma_1\ldots\sigma_n$ (where $\sigma_i$ 
is a move \textit{from} $\pi_i$ \textit{to} $\rho_i$)
is \textit{legal} if $\rho_i=\pi_{i+1}$ holds for every $i$. We call $\pi_1$ the
\textit{initial} position and $\rho_n$ and the \textit{resulting} position of $\sigma$.
Moreover, we say that $\sigma$ is a \textit{repetition} if $\sigma$ is legal, $n=2$, and $\pi_1=\rho_2$.

\section{An auxiliary game}
We need to introduce a new game in order to give a solution for Problem~\ref{probexp}.
Consider a graph $G$ which is a union of a cycle of length $3m$ and $m$ triangles;
we denote the $j$th vertex of $i$th triangle by $(i,j)$ with $i\in\{1,\ldots,m\}$
and $j\in\{1,2,3\}$. The vertices of the large cycle are labeled $(0,t)$
with $t\in\mathbb{Z}/3m\mathbb{Z}$. The game is played with two types of pieces, one
\textit{chip} and $3m$ \textit{switches}. 
We assume that every vertex of $G$ is
either empty, or occupied by the chip, or occupied by exactly one switch.
(In particular, it is not possible for a vertex to be occupied by more than one
switch at the same time, or by the chip and a switch at the same time.) 
Also, we assume that either $(i,j)$ or $(0,3i+j-3)$ is occupied by a switch.
The game consists in completing a legal sequence of \textit{single moves}
each of which belongs to one of the following categories.

\textit{Swap.} If the chip is located in one of the vertices $(i,j)$ and $(0,3i+j-3)$,
it can be swapped with a switch from the other of these vertices.

\textit{Jump.} If the chip is located on a triangle, it can jump (that is, can be moved)
to any vacant vertex of the same triangle. If $(0,j)$ is occupied by the chip, it can
jump to $(0,j+1)$ or $(0,j-1)$ if any of these vertices is vacant.

In initial position $P$, switches occupy vertices $(i,1)$, $(i,3)$, $(0,3i-1)$;
the chip is located initially in arbitrary vertex
on the large cycle. Figure~\ref{figinitaux} represents the starting position with $m=3$ with
dark squares being switches and \textit{C} a chip.

\begin{figure}[tbph]
  \begin{center}
      \includegraphics[scale=0.38]{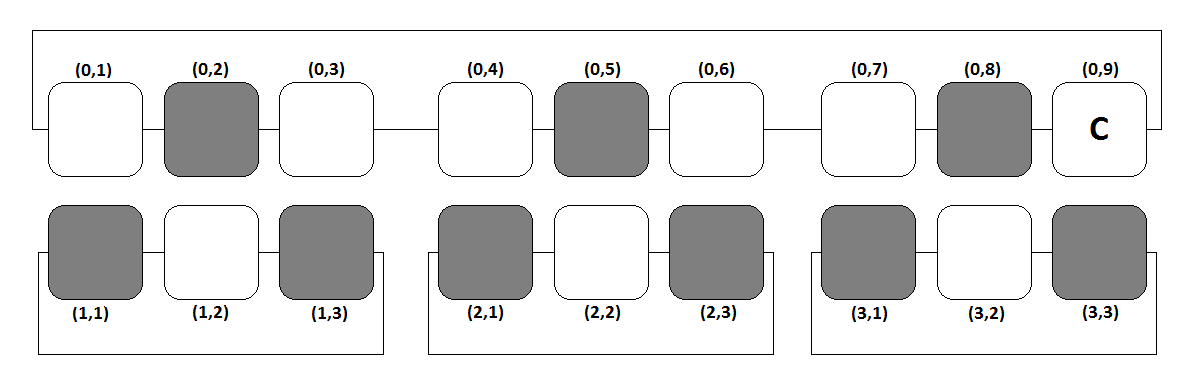}
    \caption{Initial position of auxiliary game.}
    \label{figinitaux}
  \end{center}
\end{figure}

Consider a move sequence $\sigma$ resulting in a position $P'$. Denote
by $\#(\sigma,i_0,j_1,j_2)$ the number of single moves from
$\sigma$ that are jumps from $(i_0,j_1)$ to $(i_0,j_2)$.
(In particular, $\#(\sigma,0,j_1,j_2)$ may be nonzero
only if $j_1-j_2$ is either $1$ or $-1$.)
Define the \textit{flow} $\mathcal{F}(\sigma,i_0,j_0)$ through $(i_0,j_0)$ as the half of
$$\#(\sigma,i_0,j_0-1,j_0)+\#(\sigma,i_0,j_0,j_0+1)-
\#(\sigma,i_0,j_0+1,j_0)-\#(\sigma,i_0,j_0,j_0-1).$$

\begin{lem}\label{lemflow}
If $\sigma$ is legal and $P=P'$, then $\mathcal{F}(\sigma,i,2)=\mathcal{F}(\sigma,0,3i-1)$.
\end{lem}

\begin{proof}
Focus the attention on those single moves from $\sigma$ which involve vertices $(i,2)$ and $(0,3i-1)$,
and those which change the positions of switches acting between $(i,j)$ and $(0,3i+j-3)$; we call these
moves \textit{effective}, and all other moves \textit{ineffective}. By $\sigma'$ we denote the subsequence
of $\sigma$ consisting of all effective moves; note that $\sigma'$ need not be legal. By definition of flow,
$\sigma$ and $\sigma'$ have the same flows through $(i,2)$ and the same flows through $(0,3i-1)$.
The combinatorial analysis shows that any effective move leads from a position
indexed $t$ on Figure~\ref{fig12} to the position indexed either $t+1$ or $t-1$ (such a move is called a \textit{direct type
$t$} move in the former case and a \textit{reversed type $t$} move in the latter case).

\begin{figure}[tbph]
  \begin{center}
\includegraphics[scale=0.34]{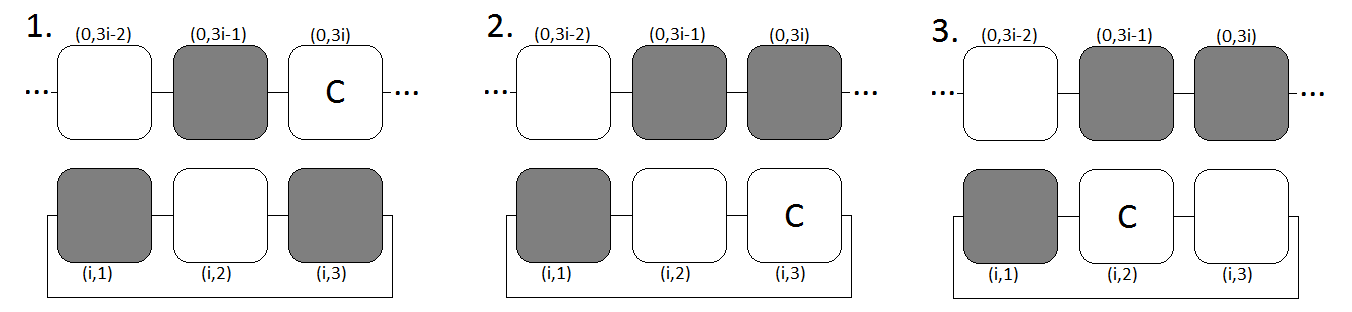}
\includegraphics[scale=0.34]{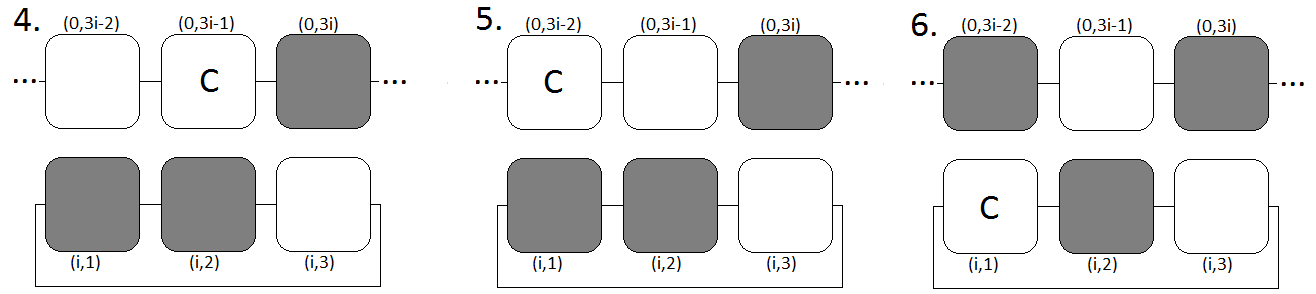}
\includegraphics[scale=0.34]{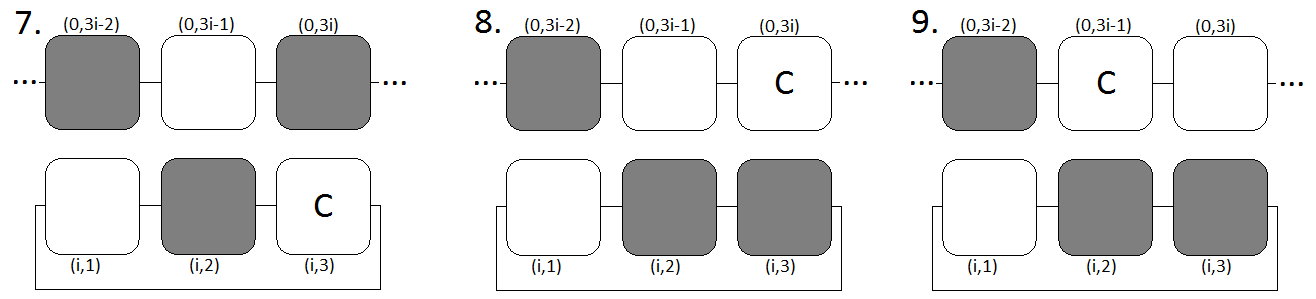}
\includegraphics[scale=0.34]{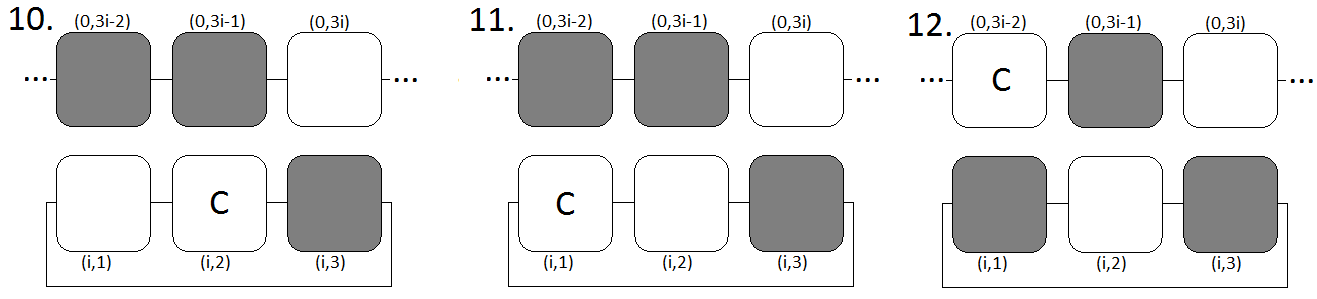}
    \caption{Positions corresponding to effective moves.} 
    \label{fig12}
  \end{center}
\end{figure}

Note that a legal sequence of ineffective moves cannot lead from a position indexed $t_1$ to a position indexed $t_2$ unless $t_1=t_2$
or $\{t_1,t_2\}=\{1,12\}$. Therefore, a move of direct type $t$ (modulo $11$) can be followed in $\sigma'$ by a type $t+1$ move only;
similarly, a move of reversed type $t$ can be followed by a type $t-1$ move only.
Since a direct type $t$ move and a reversed type $t+1$ move form a sequence with zero flow, 
we can remove such a pair from $\sigma'$ without changing its flow.
Therefore, it suffices to consider the case when $\sigma'$ is a sequence of $11$ consecutive direct type $1,\ldots,11$ moves
(or inverse type $12,\ldots,2$ moves). In this case, the quantities
$\mathcal{F}(\sigma',i,2)$ and $\mathcal{F}(\sigma',0,3i-1)$ equal $-1$ (or $1$, respectively).
\end{proof}

Let us complete the legal sequence of effective moves leading from 1 to 12, then move the chip from $(0,3i-2)$ to $(0,3i-3)$.
Then, we can do these moves again for switches between $(i-1,j)$ and $(0,3(i-2)+j)$ instead of those between $(i,j)$ and $(0,3i+j-3)$.
By induction, we construct a legal sequence $\tau$ of moves which leads to the starting position and satisfies $\mathcal{F}(\tau,i,2)=1$, for any $i$.
Completing the sequence $\tau$ several times still leaves us in the starting position, and we have $\mathcal{F}(\tau^k,i,2)=k$.


\section{The chess positions}
Define a \textit{short bishop} graph on the $n\times n$ chessboard by declaring a pair
of squares adjacent if they have the same color and one of them
can be reached from the other by a single king move. A subset of
chessboard is called a \textit{bishop cycle} if the short bishop graph it induces is a cycle.
Removing two non-adjacent squares splits a bishop cycle into two connected components which we call \textit{segments}.
We say that a pair of squares is \textit{touching} if they have different colors and one of them
can be reached from the other by a single king move.

Now we are ready to provide a pair of positions $P$ and $P'$ on the $n\times n$ chessboard such that
(1) there is a legal sequence $\sigma$ of chess moves leading from $A$ to $B$; (2) the length of $\sigma$
cannot be less than some function growing exponentially in $n$. The positions $A$ and $B$ will use the same
multiset of pieces, and the pawns in $A$ will be located in the same positions as in $B$. Therefore, $\sigma$
can contain no captures and no pawn movements. Also, removing any repetition contained in $\sigma$, we get a
legal sequence of smaller length with the same starting and initial positions; therefore, we can assume that
$\sigma$ contains no repetitions.

The only pieces actually used in our positions are pawns, bishops, and rooks. These pieces are located in a region $R$
of the board bounded by a pawn chain. Since the pawns are untouchable, this chain forbids the pieces to leave $R$.
We can add $O(1)$ more ranks and files to $R$ and locate there a pair of kings, in order to make our positions satisfy
the rules of chess. These settings allow us to consider a chess-like game played in region $R$ with standard chess rules
except that (1) White and Black do not have to alternate their moves, (2) captures, pawn moves, and repetitions are forbidden,
and (3) there are no kings.

In order to describe the position within $R$, denote $p_1=7$ and let $p_{i+1}$ be the smallest prime exceeding $p_i$.
Let $R$ contain $m$ disconnected dark-squared bishop cycles of lengths $2(p_1+1),\ldots,2(p_m+1)$
and one light-squared bishop cycle; all other squares in $R$ are occupied by pawns. There are $3m$ pairwise non-adjacent
squares on light cycle which we call \textit{switch squares} and identify with vertices $(0,t)$ from $G$; we assume that one of the segments
between $(0,t)$ and $(0,t+1)$ contains no switch squares. On $i$th dark cycle, there will be also three non-adjacent \textit{switch squares}
identified with vertices $(i,1)$, $(i,2)$, $(i,3)$ from $G$. We say that a square $x$ lies \textit{between} $(i_0,j_0)$ and $(i_0,j_0+1)$
if $x$ belongs to that segment with ends $(i_0,j_0)$ and $(i_0,j_0+1)$ which contains no switch squares. We assume that the $(i,j)$ and
$(0,3i-3+j)$ switch squares are touching, and no other pair of squares on cycles is touching.

In the starting position, the switch squares $(i,1)$, $(i,3)$, $(0,3i-1)$ are occupied by rooks;
one of the squares on the light cycle is vacant, and all the other squares are occupied by bishops.
One of the bishops on every dark cycle is white, and all the other pieces are black. (In fact, there is no need to
specify the colors of pieces outside the dark cycles.) It is easy to construct such a position in the case when
$p_m<cn$, for some sufficiently small absolute constant $c$. An example, which corresponds to the case
$p_1=7$, $p_2=11$, $p_3=13$, is provided on Figure~\ref{figchess}; the dotted squares correspond to pawns, which are untouchable in our model.

\begin{figure}[tbph]
  \begin{center}
      \includegraphics[scale=0.4]{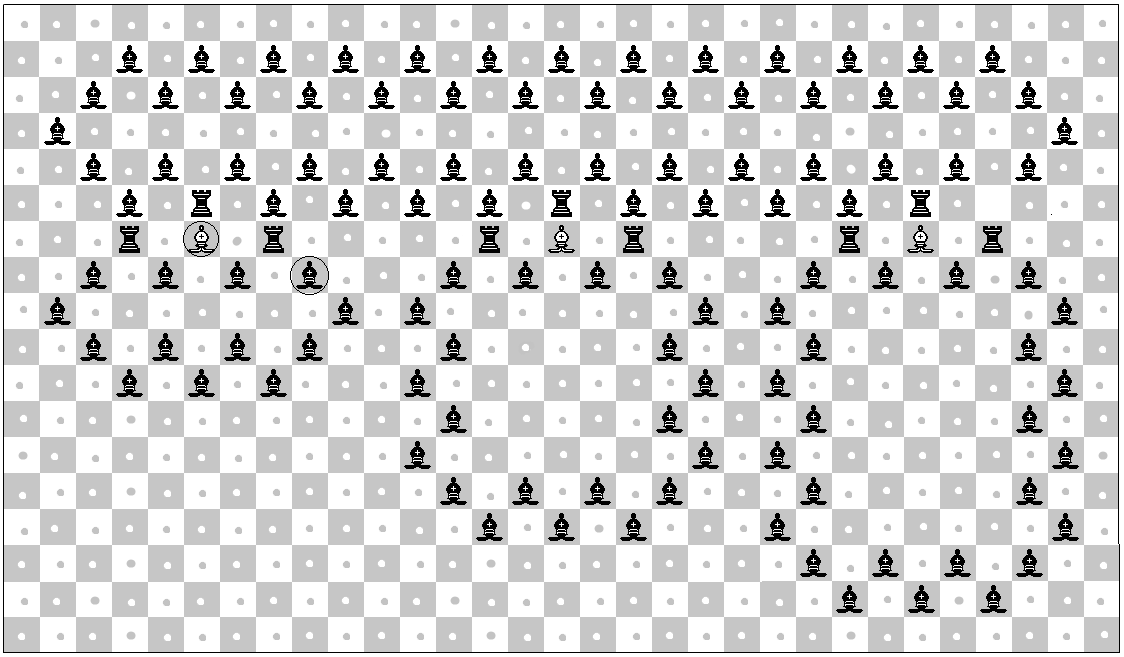}
    \caption{Position $P$. A swap of marked bishops leads to $P'$.} 
    \label{figchess}
  \end{center}
\end{figure}

Now we explain how is the constructed position related to the game described in previous section.
Note that every rook can only move between the $(i,j)$ and $(0,3i-3+j)$ switch squares.
Since repetitions are not allowed, we see that moving a bishop located between switch
squares $(i_0,j_0)$ and $(i_0,j_0+1)$ to the switch square $(i_0,j_0+1)$ is to be followed
by the sequence of consecutive bishop moves eventually leaving $(i_0,j_0)$ vacant. Therefore, playing chess
in the constructed position is the same as playing the game from previous section as we identify
rooks with switches and the vacant square, which is unique in region $R$, with chip.

Now let $\sigma$ be any legal sequence of moves such that the starting and resulting positions coincide, up to color of pieces.
We can note that the flow $\mathcal{F}(\sigma,0,j)$ is equal to the number of times a bishop located between switch
squares $(0,j)$ and $(0,j+1)$ moved through $(0,j)$ to a position between $(0,j)$ and
$(0,j-1)$ minus the number of times bishops moved in opposite direction. In particular, it is now clear that
the flow $\mathcal{F}(\sigma,0,t)$ does not depend on $t$. Similarly, since there are one white and $2p_i-1$ black bishops
moving in $i$th cycle, the remainder of $\mathcal{F}(\sigma,i,2)$ modulo $2p_i$ determines the position of white bishop.
By Lemma~\ref{lemflow}, the flow $\mathcal{F}(\sigma,i,2)$ is independent of $i$, but it can take any value as $\sigma$ varies.

Now consider the starting position $P$ and construct a new position $P'$ by swapping the positions of the white bishop from the first
dark cycle and a black bishop located a distance two apart from this white bishop.
(In particular, $P'$ can be obtained from the position on Figure~\ref{figchess} by swapping the bishops contained in ovals; in general,
we say that two bishops from the same dark cycle lie a \textit{distance two apart} if they define a segment containing exactly one bishop.)
Then, we conclude by the Chinese remainder theorem that there is a sequence $\sigma'$ leading from $P$ to $P'$.
Then, $\sigma'$ is such that $\mathcal{F}(\sigma',i,2)$ divides $2p_i$ if and only if $i\neq1$. Now we see that
$\mathcal{F}(\sigma',i,2)$ is non-zero and divides $\prod_{i=2}^mp_i$; the latter quantity is exponential in $n$
since the product of primes not exceeding $n$ is $\exp(n+o(n))$.

\section{Concluding remarks}
Problem~\ref{probexp} has also been studied for some less natural generalizations of chess, see a discussion in~\cite{MO}.
In particular, there was an attempt of solving it for a chess-like game with $O(n)$ kings none of which can be leaved under
attack. Also, there is a negative solution for Problem~\ref{probexp} if a game is assumed to terminate 
after fifty moves without captures and pawn moves; this generalization, however, seems rather artificial.
In fact, the rules of $8\times8$ chess do not say that a game terminates immediately after
fifty consecutive moves of this kind: The players are allowed to play on
if neither of them wants a draw~\cite{ChessLaws}. Finally, note that some authors~\cite{Chess} consider an
upper bound for number of pieces growing as a fractional power of board size. In order to
make our positions satisfy this restriction, we can add sufficiently many empty ranks and files to them,
getting a lower bound for the diameter of $n\times n$ chess exponential in a fractional power of $n$.

\bigskip

I would like to thank Timothy Y. Chow for interesting discussion and many helpful suggestions on presentation of the result.

\end{document}